\documentclass{amsart}
\usepackage{geometry} % see geometry.pdf on how to lay out the page. There's lots.
\usepackage{graphicx, amsfonts, amssymb, amsthm,enumerate, setspace}
\newtheorem{theorem}{Theorem}

\newtheorem{lemma}{Lemma}
\newtheorem{proposition}{Proposition}

\newtheorem{definition}{Definition}

\newcommand{\Le}{\mathfrak{Le}}
\newcommand{\seq}[1]{\left\{ #1 \right\}_{k=0}^{\infty}}

\newcommand{\set}[1]{\left\{#1\right\}}
\newcommand{\R}{\mathbb R}
\newcommand{\C}{\mathbb C}
\newcommand{\N}{\mathbb{N}}

\newcommand{\ds}[1]{\displaystyle{ #1}}

\newcommand{\vp}{\varphi}
\newcommand{\LP}{\mathcal{L}-\mathcal{P}}

\title{On Legendre Multiplier Sequences}
\author{Kelly Blakeman${}^{\dag}$}
\address{${}^{\dag}$-Department of Mathematics\newline \indent
Loyola Marymount University \newline \indent 1 LMU Drive, Suite 2700, University Hall \newline \indent Los Angeles, CA 90045}
\email{kblakema@lion.lmu.edu}
\author{Emily Davis${}^{\ddag}$} 
\address{${}^{\ddag}$-Department of Mathematics\newline \indent
Brigham Young University \newline \indent 275 TMCB \newline \indent Provo, UT 84602}
\email{emilydavis89@byu.edu}
\author{Tam\'as Forg\'acs${}^{\S}$}
\address{${}^{\S}$- corresponding author \newline \indent Department of Mathematics\newline \indent
California State University, Fresno \newline \indent 5245 North Backer Ave., M/S PB 108 \newline \indent Fresno, CA 93740-8001}
\email{tforgacs@csufresno.edu}
\author{Katherine Urabe${}^{\star}$}
\address{${}^{\star}$-Department of Mathematics\newline \indent
California State University, Fresno \newline \indent 5245 North Backer Ave., M/S PB 108 \newline \indent Fresno, CA 93740-8001}
\email{kturabe@mail.fresnostate.edu}

\date{}% delete this line to display the current date

\begin{document}
\maketitle

\begin{abstract}
In this paper we give a complete characterization of linear, quadratic, and geometric Legendre multiplier sequences. We also prove that all Legendre multiplier sequences must be Hermite multiplier sequences, and describe the relationship between the Legendre and generalized Laguerre multiplier sequences. We conclude with a list of open questions for further research.  {\bf 26C10, 30C15}

\smallskip

\noindent {\it Keywords:} Multiplier sequence, Legendre polynomials, reality preserving operators.
\end{abstract}

%%%%%%%%%%%%%%%%%%%%%%%%%%%%%INTRODUCTION%%%%%%%%%%%%%%%%%%
%\begin{doublespace}
\section{Introduction}

A set of polynomials $Q=\{ q_k \}_{k=0}^{\infty}$ is called {\it simple} if $\deg q_k=k$ for all $k \in \N \cup \set{0}$.  Given a simple set of polynomials $Q=\{q_k \}_{k=0}^{\infty}$ and a sequence of real numbers $\{\gamma_k\}_{k=0}^{\infty}$, one can define a linear operator $T$ on $\mathbb{R}[x]$ by declaring $T[q_k(x)]=\gamma_k q_k(x)$ for all $k \in \N \cup \set{0}$. We call $\{\gamma_k\}_{k=0}^{\infty}$ a \textit{Q-multiplier sequence} if $T[p]$ has only real zeros whenever $p$ has only real zeros.  In the case when $Q$ is the standard basis, we follow the existing literature by using the terminology `multiplier sequence' or `classical multiplier sequence' without a reference to $Q$.
\newline \indent  Whether a sequence $\seq{\gamma_k}$ is a $Q$-multiplier sequence depends crucially on the choice of the set $Q$. In \cite{p} Piotrowski has shown that every $Q$-multiplier sequence is a classical multiplier sequence if $Q$ is any simple set of polynomials (see Theorem \ref{qmscms} in Section \ref{background}). There has been recent progress in giving conditions under which multiplier sequences for a simple set $Q$ are also multiplier sequences for another simple set $\widetilde{Q}$, with $\widetilde{Q}$  not necessarily the standard basis (see \cite{tjb}), although the theory in this much generality is still incomplete.
In this paper, we focus our attention on the simple set of Legendre polynomials and their corresponding multiplier sequences.  
\begin{definition} \label{legendredef}The Legendre polynomials $\Le_n(x)$ are defined by the following generating function
\[
\frac{1}{\sqrt{1-2xt+t^2}}=\sum_{k=0}^{\infty} \Le_k(x)t^k.
\]
\end{definition}

The choice of the Legendre polynomials is motivated by the fact that the Legendre and Hermite polynomials are both defined using generating functions of the form
\[
G(2xt-t^2) = \displaystyle\sum_{k=0}^{\infty} g_k(x)t^k.
\]
As a consequence, the Legendre and Hermite polynomials satisfy very similar differential equations (see \cite{rainville}, p. 132), suggesting that Legendre and Hermite multiplier sequences might be closely related. Since the Hermite multiplier sequences have been completely characterized by Piotrowski in \cite{p}, we had hoped to achieve a similar result for the Legendre multiplier sequences.\\
\indent The rest of the paper is organized as follows. Section \ref{background} gives a brief review of relevant results in the literature. In Section \ref{prop} we investigate the properties of Legendre multiplier sequences and show that the set of Legendre multiplier sequences is a subset of the Hermite multiplier sequences. Section \ref{polyint} contains the classification of all linear, quadratic, and geometric Legendre multiplier sequences. Section \ref{concl} concludes with some open questions. 

%%%%%%%%%%%%%%%%%%%BACKGROUND%%%%%%%%%%%%%%

\section{Background} \label{background}

In the late 1800s Laguerre and Jensen were already investigating the existence of classical multiplier sequences, but it was not until 1914 that a complete characterization of all such sequences would emerge in a paper by P\'olya and Schur (\cite{ps}).
\begin{definition} A real entire function $\ds{\vp(x)=\sum_{k=0}^{\infty} \frac{\gamma_k}{k!}x^k}$ is said to belong to the Laguerre-P\'olya class $\LP$ if and only if it is the locally uniform limit in $\C$ of real polynomials having only real zeros\footnote{The Laguerre-P\'olya class is usually defined as a set of functions with a particular Weierstrass factorization. For the sake of exposition we opted for this simpler, but equivalent definition}. If, in addition, $\gamma_k \geq 0$ for $k=0,1,2,\ldots$, we will write $\vp \in \LP^+$.
\end{definition}
\begin{theorem}\label{PolyaSchur}
\emph{(P\'olya-Schur \cite{ps})} \label{tc} Let $\{\gamma_k\}_{k=0}^{\infty}$ be a sequence of non-negative real numbers.  The following are equivalent:
\begin{enumerate}
\item{$\{\gamma_k\}_{k=0}^{\infty}$ is a multiplier sequence}.
\item For each $n$, the polynomial $\displaystyle T[(1+x)^n]:= \sum_{k=0}^{n}\binom{n}{k} \gamma_k x^k \in \LP^+$.
\item $\ds{ T[e^x]:= \sum_{k=0}^{\infty} \frac{\gamma_k}{k!}x^k} \in \LP^+$.
\end{enumerate}
\end{theorem}
The following theorem of Piotrowski relates multiplier sequences for any simple set $Q$ to classical multiplier sequences:
\begin{theorem}[Piotrowski, 2007] \label{qmscms} Let $Q=\seq{q_k(x)}$ be any simple set of polynomials. If the sequence $\seq{\gamma_k}$ is a $Q$-multiplier sequence, then it is also a classical multiplier sequence.  
\end{theorem}
We shall make use of Theorem \ref{qmscms} repeatedly as we look for properties every Legendre multiplier sequence has to satisfy. \\
As we mentioned in the introduction, if $\seq{q_k(x)}$ is a simple set of real polynomials and $\seq{\gamma_k}$ is any sequence of real numbers, then the operator defined by $T[q_k(x)]=\gamma_kq_k(x)$ for all $k \in \N \cup \set{0}$ is a linear operator on the polynomial ring $\R[x]$. The following theorem of Piotrowski guarantees that every linear operator on $\R[x]$ has a unique differential operator representation.

\begin{theorem}\label{diffoprep}
Let $T: \R[x] \to \R[x]$ be a linear operator. Then there exists a unique set of complex polynomials $\{p_k(x)  \}_{k=0}^{\infty}$ such that 
\begin{equation*}\label{linoprepeq}
T[f(x)]=\sum_{k=0}^{\infty} p_k(x) f^{(k)}(x)
\end{equation*}
for all $f \in \R[x]$.
\end{theorem}
\smallskip
\noindent{\bf Example 1.} Consider the sequence $\Gamma=\seq{k}$. Since $\ds{x (x^k)' =kx^k}$ for $k=0,1,2,\ldots$, we see that this sequence has the differential operator representation $\Gamma=xD$, where $D$ denotes differentiation with respect to $x$. Rolle's theorem implies that differentiation preserves the reality of zeros. Since multiplication by $x$ only introduces another zero at 0, it follows that $xD$ is a reality preserving operator. Thus $\seq{k}$ is a classical multiplier sequence. 

\smallskip
Should we choose to study the reality preserving properties of a sequence through its differential operator representation, we need to be able to decide whether a given differential operator is reality preserving. We have a deep result of Borcea and Br\"and\'en from 2009 to aid us in this endeavor.
\begin{theorem}[Borcea-Br\"{a}nd\'{e}n \cite{BB}]\label{3exw}
A linear operator $T: \mathbb{R}[x] \rightarrow \mathbb{R}[x]$ preserves reality of zeros if and only if either
\begin{enumerate}
\item{T has range of dimension at most two and is of the form $T[f] = \alpha(f)P+ \beta(f)Q$, where $\alpha$ and $\beta$ are linear functionals on  $\mathbb{R}[x]$, and $P$ and $Q$ are polynomials with only real interlacing zeros, or}
\item{$T[e^{-xw}]=\displaystyle\sum_{k=0}^{\infty} \frac{(-w)^nT[x^n]}{n!} \in \overline{A}$, or}
\item{$T[e^{xw}]=\displaystyle\sum_{k=0}^{\infty} \frac{w^nT[x^n]}{n!} \in \overline{A}$}
\end{enumerate}
where $\overline{A}$ denotes the set of entire functions in two variables which are uniform limits on compact subsets of polynomials in the set
\begin{center}
$A=\{f \in \mathbb{R}[x,w] | f(x,w) \neq 0$ whenever Im $(x)>0$ and Im $(w)>0\}$.
\end{center}
\end{theorem}  

We close this section by recalling two theorems regarding the reality of zeros of cubic and quartic polynomials for the reader's convenience. We make heavy use of these theorems in Section \ref{polyint}, when we investigate Legendre multiplier sequences interpolated by polynomials.
\begin{theorem}\label{cd}(\cite{quartic}, p. 154)
Let $f(x)=ax^3+bx^2+cx+d$. Consider the discriminant of $f(x)$,  \\ $\Delta=b^2c^2-4b^3d-4ac^3+18abcd-27a^2d^2$.   If
\begin{enumerate}
\item[(1)]{$\Delta > 0$, then $f$ has all real roots}
\item[(2)]{$\Delta < 0$, then $f$ has one real root and two complex conjugate roots}.
\end{enumerate}
\end{theorem}

\begin{theorem}\label{qd}(\cite{quartic}, p. 167-170)
Let $g(x)=ax^4+bx^3+cx^2+dx+e$ be a quartic function, where $a, b, c, d, e \in \R$. Consider the discriminant of $g(x)$. 
 If
\begin{enumerate}
\item{$\Delta > 0$, then $g$ has either all real or all complex roots}
\\ \noindent Consider the depressed quartic $h(x)=z^4+qz^2+rz+s$.  If
\begin{enumerate}
\item{$q<0$ and $q^2-4s >0$, then the roots of the cubic resolvent are all positive and the roots of the given quartic are all real}
\item{$q<0$ and $q^2-4s >0$ do not both hold, then only one root of the cubic resolvent is positive and no roots are real}
\end{enumerate}
\item{$\Delta = 0$, then there may or may not be complex roots}
\item{$\Delta < 0$, then $g$ has two real roots and two complex conjugate roots}.
\end{enumerate}
\end{theorem}
 
 %%%%%%%%%%%%%%%%%%%%%PROPERTIES%%%%%%%%%%%%%%%%%%
\section{Properties of Legendre Multiplier Sequences} \label{prop}
We begin this section with the definition of a Legendre multiplier sequence. 
\begin{definition}
Let $\seq{\gamma_k}$ be a sequence of real numbers. If
\[
\displaystyle\sum_{k=0}^{n} a_k \gamma_k \Le_k(x) 
\]
has only real zeros whenever 
\[
p(x)= \displaystyle\sum_{k=0}^{n} a_k \Le_k(x) 
\]
has only real zeros, we say that $\seq{\gamma_k}$ is a Legendre multiplier sequence.
\end{definition}
In the introduction we defined the Legendre polynomials via a generating function. Alternatively, the $n^{th}$ Legendre polynomial $\Le_n(x)$ can also be obtained from Rodrigues' formula (see for example p.162 \cite{rainville}):
\begin{equation}\label{rod}
\Le_n(x)=\frac{1}{2^{n} n!}D^n[(x^2-1)^n].
\end{equation}
\noindent An immediate consequence of equation (\ref{rod}) is that $\deg \Le_n(x)=n$. Therefore the Legendre polynomials form a simple set. In addition, 
\[
\int_{-1}^1 \Le_n(x) \Le_m(x)dx=\left\{\begin{array}{cl} 0 & \textrm{if} \ n \neq m \\ \ds{\frac{2}{2n+1}} & \textrm{if} \ n=m\end{array} \right. ,
\]
and hence they also form an orthogonal set on the interval $[-1,1]$. Orthogonality and simplicity of $\seq{\Le_k(x)}$ has profound consequences regarding the zeros of the Legendre polynomials (see \cite{sz}, p. 43-45). In particular: \\

\noindent (i) $\Le_n(x)$ has $n$ simple real zeros in $[-1,1]$ for $n=0,1,2, \ldots$\\
(ii) $\Le_n(x)$ and $\Le_{n-1}(x)$ have interlacing zeros for $n=1,2, 3, \ldots$ \\
(iii) $a \Le_n(x) + b \Le_{n-1}(x)$ has only real zeros for any $a, b \in \R$, $n=0,1,2, \ldots$\\

\smallskip
Property (iii) above says that any sequence of the form $(\ldots,0,0, a,b, 0, 0, 0, \ldots), \ a,b \in \R$ is a Legendre multiplier sequence. In addition to these, every constant sequence is also a Legendre multiplier sequence. We refer to these two types of sequences as {\it trivial Legendre multiplier sequences}. Unless explicitly stated otherwise, In the remainder of this paper we only consider non-trivial multiplier sequences.\\
\indent Since the Legendre polynomials form a simple set, by Theorem \ref{qmscms} every Legendre multiplier sequence is also a classical multiplier sequence. As such, Legendre multiplier sequences inherit a list of properties from the classical multiplier sequences, which we list in the next lemma.

\begin{lemma}\label{props}
Let $\seq{\gamma_k}$ be a Legendre multiplier sequence.  The following statements hold:
\begin{enumerate}
\item[(i)] If there exists integers $n> m \ge 0$ such that $\gamma_m \neq 0$ and $\gamma_n =0$, then $\gamma_k =0$ for all $k \ge n$.
\item[(ii)] The terms of $\seq{\gamma_k}$ are either all of the same sign, or they alternate in sign.
\item[(iii)] The terms of $\seq{\gamma_k}$ satisfy Tur\'{a}n's inequality:
\[
\gamma_k^2 -\gamma_{k-1} \gamma_{k+1} \ge 0 \hskip20pt (k=1,2,3, \ldots).
\]
\end{enumerate}
\end{lemma}

The remainder of this section is dedicated to showing that the set of Legendre multiplier sequences form a strict subset of the Hermite multiplier sequences. A similar result was proved by Forg\'{a}cs and Piotrowski in \cite{fp}. In the interest of self-containment, we outline their argument here, mostly without proof, noting that we merely made the necessary changes to obtain the result about Legendre multiplier sequences. 

\begin{lemma}\label{openclosed}
Let $p$ and $q$ be real polynomials with $\deg(q)<\deg(p)$.  
\begin{itemize}
\item[(i)] If $p$ has only simple real zeros then there exists $\epsilon>0$ such that $p(x)+ b q(x)$ has only real zeros whenever $|b|<\epsilon$. 
\item[(ii)] If $p$ has some non-real zeros then there exists $\epsilon>0$ such that $p(x)+ b q(x)$ has some non-real zeros whenever $|b|<\epsilon$. 
\end{itemize}
\end{lemma}

\begin{lemma}\label{bmax}
For $n\geq 2$ and $b\in\mathbb{R}$, define
\[
f_{n, b, \alpha} (x) := \Le_n(x)+b \Le_{n-2}(x), \hskip3pt \text{and}
\]
\[
E_n:=\{ b \in \mathbb{R} \ | \ f_{n, b, \alpha}(x) \hskip 5pt \text{has only real zeros}\}.
\]
%\end{eqnarray*}
Then $\max(E_n)$ exists, and is a positive real number.
\end{lemma}

%\begin{proof}
%By Lemma $\ref{openclosed}$ part $(i)$, there exists $\epsilon>0$ such that $(-\epsilon, \epsilon)\subseteq E_n$.  In particular, $E_n$ is nonempty and $\max(E_n)$, if it exists, is positive.  It now suffices to show that $E_n$ is closed and bounded above.  

%Suppose $t \in (\mathbb{R}\setminus E_n)$.  Then, by part $(ii)$ of Lemma $\ref{openclosed}$, there exists $\delta>0$ such that 
%\[
%f_{n, t, \alpha}(x)+ b \Le_{n-2}(x)  = \Le_n(x)+(t+b) \Le_{n-2}(x)
%\] 
%has non-real zeros whenever $|b|<\delta$.  That is to say, $(b-\delta, b+\delta)\subseteq (\mathbb{R}\setminus E_n)$.  Whence, $\mathbb{R}\setminus E_n$ is open and, therefore, $E_n$ is closed.

%To show that $E_n$ is bounded above, we consider the $(n-2)^{nd}$ derivative of $f_{n,b,\alpha}$.  A calculation shows that
%\[
%\frac{d^{n-2}}{dx^{n-2}} f_{n, b, \alpha} (x) = \frac{(2n)!}{2^{n+1} n! }x^2 -\frac{n(2n-2)!}{2^n n!} + b \frac{(2n-4)!}{2^{n-2}(n-2)!}.  
%\]
%Thus $\displaystyle \frac{d^{n-2}}{dx^{n-2}} f_{n, b, \alpha} (x)$, and therefore $ f_{n, b, \alpha} (x)$, has some non-real zeros whenever $b$ is sufficiently large.
%\end{proof}

\begin{proposition} \label{gammajunk} Suppose that $\{\gamma_k\}_{k=0}^{\infty}$ is a non-trivial Legendre multiplier sequence. Then there exists an $m \in \mathbb{Z}$ such that $\gamma_k=0$ for all $k<m$ and $\gamma_k \neq 0$ for all $k \geq m$. 
\end{proposition}

%\begin{proof} 
%%%%%%%%%%%NEED TO REWORD!!!!!!!!!!!!!!!!!!!!!!!!
%Since $\{\gamma_k\}_{k=0}^{\infty}$ is a non-trivial multiplier sequence, there is at least one $k \in \mathbb{Z}$ such that $\gamma_k \neq 0$. Let $m$ be the minimal index such that $\gamma_m \neq 0$. It is easy to see that $\gamma_{m+1}$ and $\gamma_{m+2}$ are non-zero, for if either of them were zero, by part (i) of Lemma \ref{props} we would conclude that $\{\gamma_k\}_{k=0}^{\infty}$ is a trivial multiplier sequence. Suppose now that there exists a $n>m+2$ such that $\gamma_n=0$. By Lemma $\ref{bmax}$, there are constants $a_m, a_{m+2}$ such that the polynomial
%\[
%\widetilde{q}{(x)}=a_m\gamma_m \Le_m(x)+a_{m+2}\gamma_{m+2} \Le_{m+2}(x)
%\]
%has some non-real zeros.  On the other hand, by Lemma $\ref{openclosed}$, there exists $a_n$ such that 
%\begin{eqnarray*}
%q(x)&=&a_m \Le_m(x)+a_{m+2} \Le_{m+2}(x)+a_n \Le_n(x)\\
 %&=& a_n\left(\Le_n(x) + \frac{a_m}{a_n} \Le_m(x) + \frac{a_{m+2}}{a_n} \Le_{m+2}(x)\right)\\  
%\end{eqnarray*}
%has only real zeros.  Applying the Legendre multiplier sequence $\{\gamma_k\}_{k=0}^{\infty}$ to $q(x)$, we obtain the polynomial $\widetilde{q}(x)$, a contradiction. Hence $\gamma_k \neq 0$ for all $k \geq m$. 
%\end{proof}

\begin{theorem}\label{inclusion}  If the sequence of non-negative real numbers $\displaystyle{\{\gamma_k\}_{k=0}^{\infty}}$ is a non-trivial Legendre multiplier sequence, then $\gamma_k \leq \gamma_{k+1}$ for all $k \geq 0$.
\end{theorem}

\begin{proof} 
Let $T_\Le$ denote the operator associated to the Legendre multiplier sequence $\{\gamma_k\}_{k=0}^{\infty}$. Suppose $n\geq 2$ and that $\gamma_{n-2}\neq 0$. By Proposition \ref{gammajunk}, we have $\gamma_n \neq 0$.  Using the notation of Lemma \ref{bmax}, the function
\[
f_{n,\beta_n^*,\alpha}(x)=\Le_n(x)+\beta^*_n \Le_{n-2}(x) \qquad (\beta_n^* = \max(E_n))
\]
has only real zeros. It follows that 
\[
T_\Le[f_{n,\beta_n^*,\alpha}(x)]=\gamma_n \Le_n(x)+\gamma_{n-2}\beta^*_n \Le_{n-2}(x)=\gamma_{n} \left(\Le_n(x)+\frac{\gamma_{n-2}}{\gamma_n}\beta^*_n \Le_{n-2}(x) \right)
\]
also has only  real zeros. By Lemma \ref{bmax}, we must have $\displaystyle{\frac{\gamma_{n-2}}{\gamma_n}\beta^*_n \leq \beta_n^*}$, which gives $\displaystyle{0 < \frac{\gamma_{n-2}}{\gamma_n} \leq 1}$. On the other hand, by Lemma \ref{props}, we have
\[
\gamma_{n-1}^2-\gamma_{n}\gamma_{n-2} \geq 0, \quad \quad (n \geq 2),
\]
which means $\displaystyle{\left( \frac{\gamma_{n-1}}{\gamma_{n-2}}\right)^2 \geq \frac{\gamma_{n}}{\gamma_{n-2}}\geq1}$. In other words, $\gamma_{n-1} \geq \gamma_{n-2}$ and the proof is complete.
\end{proof} 

It follows that any non-trivial Legendre multiplier sequence with non-negative terms is non-decreasing. By Lemma \ref{props} we conclude that the terms of any non-trivial Legendre multiplier sequence are non-decreasing in magnitude. In \cite{p} Piotrowski proved that any classical multiplier sequence whose terms are non-decreasing in magnitude is a Hermite multiplier sequence. Since every trivial Legendre multiplier sequence is also a Hermite multiplier sequence, we have the following theorem.

\begin{theorem} The set of Legendre multiplier sequences forms a strict subset of the set of Hermite multiplier sequences.
\end{theorem}
\begin{proof} We have already established containment. To see that this containment is strict, we note that $\seq{r^k}$ is a Legendre multiplier sequence if and only if $|r|=1$ (see Theorem \ref{geom} in Section \ref{polyint}), while this sequence is a Hermite multiplier sequence for any $|r| \ge1$.
\end{proof}
We conclude this section with a diagram outlining the relationship between classical, Hermite, (generalized) Laguerre, and Legendre multiplier sequences.

\begin{center}
\includegraphics[scale=.4]{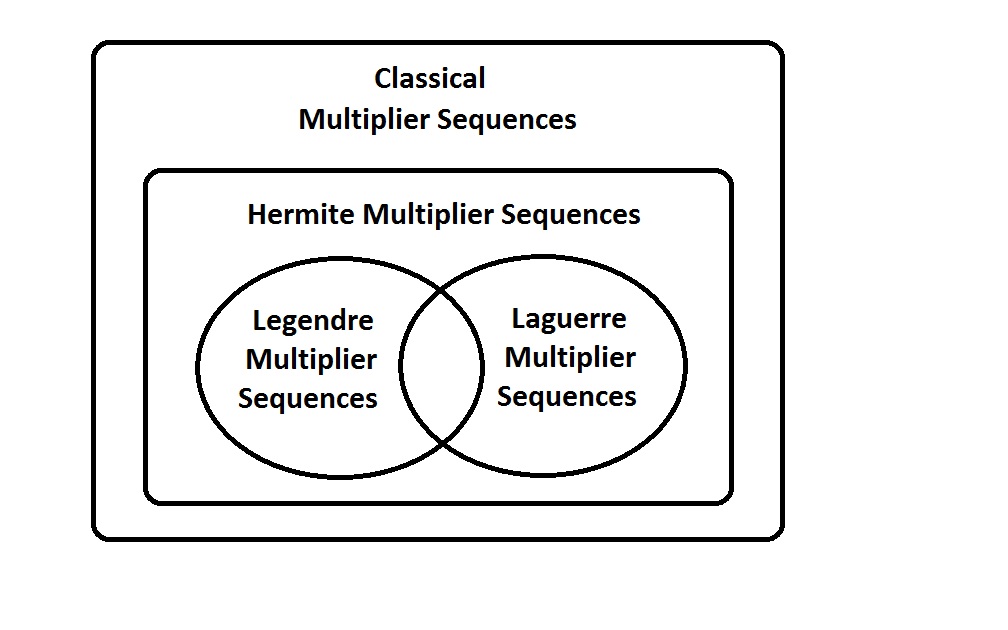}
\end{center}

%%%%%%%%%%%%LEGENDRE MS
\section{Polynomial type and Geometric Legendre Multiplier Sequences} \label{polyint}

In this section we classify linear and quadratic Legendre multiplier sequences. This line of investigation is motivated by the fact that sequences interpolated by polynomials form a large class of multiplier sequences for the standard and the Hermite bases, as the following two theorems demonstrate. The first one is due to Laguerre (\cite{cc}, p. 23), while the analogous result for the Hermite basis is due to Tur\'an (\cite{turan}, p. 289) and Bleecker and Csordas (\cite{bcs}, Theorem 2.7).

\begin{theorem} \label{cps} If $g \in \LP$ with zeros in the interval $(-\infty,0]$, then $\seq{g(k)}$ is a classical multiplier sequence. In particular, if $g$ is any real polynomial with only real, non-positive zeros , then $\seq{g(k)}$ is a classical multiplier sequence.
\end{theorem}

\begin{theorem}\label{hps} If $g \in \LP^+$, then $\seq{g(k)}$ is a Hermite multiplier sequence. In particular, if $g$ is a real polynomial with only real, negative zeros, then $\seq{g(k)}$ is a Hermite multiplier sequence. 
\end{theorem}

%%%%%%%%%%LINEAR
\subsection{Linear Sequences} An immediate consequence of Theorems \ref{cps} \& \ref{hps} is the fact that there are both classical and Hermite multiplier sequences which are interpolated by linear polynomials. As an example, we mention the sequence $\seq{k}$, which is a Hermite, and hence also classical multiplier sequence. There are also linear generalized Laguerre (or $L^{(\alpha)}$-) multiplier sequences for any $\alpha >-1$ (see \cite{fp}). In light of these results, it is somewhat surprising that there are no linear Legendre multiplier sequences.  

\begin{proposition}
$\seq{\gamma_k} = \seq{k+\alpha}$ is not a Legendre multiplier sequence for any $\alpha \in \R$.
\end{proposition}

\begin{proof}
Let $\Gamma_{\alpha}$ be the operator defined by $\bold{\Gamma}_{\alpha}[\Le_n(x)] = (n+\alpha)\Le_n(x)$ for $n=0,1,2, \ldots$, and consider the function $f(x) = (1+x)^3$ expanded in the Legendre basis: 
\[
f(x)=\frac{2}{5} \Le_3(x)+2\Le_2(x)+\frac{18}{5}\Le_1(x)+2\Le_0(x).  
\]
Applying the operator $\bf{\Gamma_{\alpha}}$ to $f(x)$ we obtain 
\begin{eqnarray*}
\bold{\Gamma}_{\alpha}[f(x)]&=&\frac{2}{5}(3+\alpha)\Le_3(x)+2(2+\alpha)\Le_2(x)+\frac{18}{5}(1+\alpha)\Le_1(x)+2\alpha\Le_0(x) \\ 
&=&\frac{1}{5}(5x^3 -3x)(\alpha+3)+(3x^2-1)(\alpha+2) + \frac{18}{5}x(\alpha +1)+2\alpha.  
\end{eqnarray*}

The discriminant of $\bold{\Gamma}_{\alpha}[f(x)]$ is given by $\Delta = -\frac{108}{125} (421+ 172\alpha +20 \alpha^2)$, which is negative for all $\alpha \in \R$.  It follows that $\bold{\Gamma}_{\alpha}[f(x)]$ has complex roots for any $\alpha \in \R$, and hence $\seq{k+\alpha}$ is not a Legendre multiplier sequence for any $\alpha \in \R$.
\end{proof}

%%%%%%%%QUAD
\subsection{Quadratic Sequences}
Similarly to the linear case, Theorems \ref{cps} \& \ref{hps} guarantee the existence of quadratic classical and Hermite multiplier sequences. In this section we classify all quadratic Legendre multiplier sequences. Recall that every quadratic Legendre multiplier sequence $\seq{k^2+\alpha k+\beta}$ is a classical multiplier sequence. This fact gives us the first restrictions on the coefficients $\alpha$ and $\beta$. 

\begin{proposition}\label{cms}
$\seq{k^2+\alpha k+ \beta}$ is a classical multiplier sequence if and only if $\alpha \ge -1$ and $0 \le \beta \le \dfrac{1}{4} (\alpha +1)^2$.
\end{proposition}

\begin{proof}
By Theorem \ref{tc}, $\seq{\gamma_k}$ is a classical multiplier sequence if and only if 
\[
T[e^x] = \displaystyle\sum _{k=0}^{\infty} \ {\frac{\gamma_k}{k!} x^k} \in \LP^+.
\]
 Let $\Gamma$ denote the operator corresponding to the sequence  $\seq{k^2+\alpha k+ \beta}$. We then have
\[
\Gamma[e^x] = \displaystyle\sum _{k=0}^{\infty} \frac{k^2+ \alpha k+ \beta}{k!} x^k =[x^2+(\alpha +1)x+ \beta]e^x,
\]
which belongs to the class $\LP^+$ if and only if $\alpha \ge -1$ and $0 \le \beta \le \dfrac{1}{4} (\alpha +1)^2$.

\end{proof}

 We show that $\seq{k^2+\alpha k+\beta}$ is a Legendre multiplier sequence if and only if $\alpha = 1$ and $\beta \in [0,1]$.  We begin by first addressing the case when $\alpha \neq 1$, after which we examine the case of $\alpha = 1$. 

\begin{proposition}\label{a1}
If $\alpha \neq 1$, then $\seq{k^2+\alpha k+\beta}$ is not a Legendre multiplier sequence for any $\beta$.
\end{proposition}

\begin{proof}
Let $\Gamma=\seq{k^2+\alpha k+\beta}$. Let $\alpha \geq -1$ and $\alpha \neq 1$.  Define 
\begin{eqnarray*}
f(\alpha,\beta,x) &=& \Gamma[(1+x)^4] \\
&=&\frac{4}{5}(5x^3-3x)(3 \alpha + \beta +9)+\frac{16}{7}(3x^2-1)(2 \alpha + \beta +4)\\
&&+\frac{1}{35}(35x^4-30x^2+3)(4 \alpha + \beta +16)+\frac{32}{5} x (\alpha + \beta +1) +\frac{16}{5} \beta,
\end{eqnarray*}
and denote the discriminant of $f$ with respect to $x$ by $\Delta_x f(\alpha, \beta,x)$ .  As a consequence of Theorem \ref{qmscms} and Proposition \ref{cms}, we need only consider $\alpha$ and $\beta$ contained in the set 
\[
A = \left\{(\alpha,\beta) \ \Big| \ \alpha \ge -1, 0 \le \beta \le \frac{(\alpha + 1)^2}{4} \right \},
\] which is shown in Figure $1$. Case $1$ examines the region shaded in Figure $2$, and Case $2$ examines the region in Figure $3$. Figure $4$ demonstrates that these cases taken together establish that $\seq{k^2+\alpha k+\beta}$ is not a Legendre multiplier sequence for any $(\alpha, \beta) \in A$ with $\alpha \neq 1$.\\

\medskip
\begin{tabular}{cc}
Figure 1&Figure 2\\
\includegraphics[scale=.52]{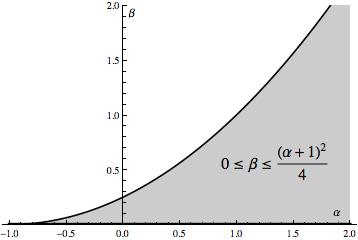}&\includegraphics[scale=.52]{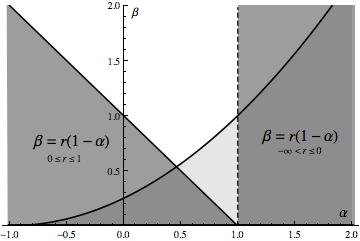}\\
Figure 3&Figure 4\\
\includegraphics[scale=.52]{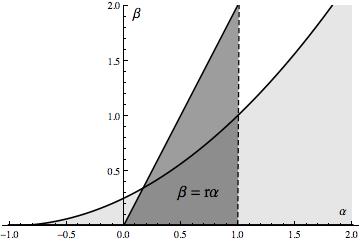}&\includegraphics[scale=.52]{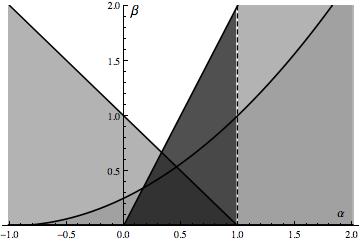} \\
\end{tabular}

\begin{itemize}
\item[\underline{Case 1:}]Let $\beta = r(1-\alpha)$ with $ r\leq 1$. As $r$ ranges through the indicated values, these lines cover the shaded area in Figure 2.\\
Given the restrictions on $\alpha$, a calculation shows that
\begin{eqnarray*}\label{final}
\Delta_r\bigg[\frac{\partial}{\partial r}\Delta_x f\Big(\alpha, r(1-\alpha), x\Big)\bigg] < 0,
\end{eqnarray*}
where $\Delta_r$ and $\Delta_x$ denote the discriminants with respect to $r$ and $x$.
It follows that the function $\ds{\frac{\partial}{\partial r} \Delta_xf\Big(\alpha, r(1-\alpha), x\Big)}$, which is quadratic in $r$, has no real zeros. We check that 
\[
\frac{\partial}{\partial r} \Delta_x f\Big(\alpha, r(1-\alpha), x\Big) \Big|_{r=0} >0
\]
and conclude that $\ds{\frac{\partial}{\partial r} \Delta_xf\Big(\alpha, r(1-\alpha), x\Big)}$ is everywhere positive. Consequently \\ $\ds{\Delta_x f\Big(\alpha, r(1-\alpha), x\Big)}$ is monotone increasing in $r$, which, together with the fact that \\$\Delta_x f\Big(\alpha, r(1-\alpha), x\Big) \Big|_{r=1} < 0$, implies that $\Delta_x f\Big(\alpha, r(1-\alpha), x\Big) < 0$ for any $r \leq 1$.  Therefore $f\Big(\alpha, r(1-\alpha),x\Big)$ has complex zeros for any $r \leq 1$ and $\alpha \geq -1, \alpha \neq 1$.

\item[\underline{Case 2:}]Let $\beta = r\alpha$ with $r \geq 0$ and $0 < \alpha < 1$. As $r$ ranges through the indicated values, these lines cover the shaded area in Figure 3.\\
A calculation shows that,
\begin{eqnarray*}
\Delta_r\bigg[\frac{\partial}{\partial r} \Delta_xf\Big(\alpha, r \alpha, x\Big)\bigg] < 0 \qquad (0<\alpha<1).
\end{eqnarray*}
It follows that the function $\ds{\frac{\partial}{\partial r} \Delta_xf\Big(\alpha, r\alpha, x\Big)}$, which is quadratic in $r$, has no real zeros.  We check that when $0 < \alpha < 1$,
\[
\frac{\partial}{\partial r} \Delta_x f\Big(\alpha, r\alpha, x\Big) \Big|_{r=0} >0
\]
and conclude that $\ds{\frac{\partial}{\partial r} \Delta_xf\Big(\alpha, r\alpha, x\Big)}$ is positive for all $r \in \R$.
 It follows that $\ds{\Delta_x f\Big(\alpha, r\alpha, x\Big)}$ is monotone increasing, which, together with the fact that $\Delta_x f\Big(\alpha, r\alpha, x\Big) \Big|_{r=2} < 0$ implies that $\Delta_x f\Big(\alpha, r\alpha, x\Big) < 0$ for any $0 \leq r \leq 2$.  Therefore $f\Big(\alpha, r\alpha,x\Big)$ has complex zeros for any $0 \leq r \leq 2$ and $0 < \alpha < 1$.
 \item[\underline{Case 3}] Let $\alpha=0$ and let $\Gamma$ be the operator corresponding to the sequence $\seq{k^2+\beta}$. The discriminant of $\Gamma[(1+x)^3]$ is negative when $\beta > 9.8149$. On the other hand, the discriminant of $\Gamma[(1+x)^4]$ is negative when $\beta < 11.7649$. These facts and Theorems \ref{cd} \& \ref{qd} imply that  $\seq{k^2+\beta}$ is not  a Legendre multiplier sequence.
\end{itemize}
It follows that when $\alpha \neq 1$, $\seq{k^2+\alpha k+\beta}$ is not a Legendre multiplier sequence for any $\beta$.
\end{proof}

We now consider the case when $\alpha =1$.  From Proposition \ref{cms} we know that if $\seq{k^2+ k+\beta}$ is a classical multiplier sequence, then $\beta \in [0,1]$.  It remains to show that if $\beta \in [0,1]$, then $\seq{k^2+ k+\beta}$ is a Legendre multiplier sequence. We first treat the cases when $\beta=0$ and $\beta =1$. Finally we deal with the case when $\beta \in (0,1)$.

\begin{lemma}\label{k2k}
$\seq{k^2+k}$ is a Legendre multiplier sequence.
\end{lemma}

\begin{proof}
The Legendre polynomials satisfy Legendre's differential equation:
\begin{eqnarray}\label{lde}
(1-x^2) \Le_n''(x) -2x\Le_n'(x) + n(n+1) \Le_n(x) =0.
\end{eqnarray}
It follows that the linear operator $\Gamma$ defined by $\Gamma[\Le_n(x)]=n(n+1) \Le_n(x)$ has the differential operator representation $T= [(x^2-1)D +2x]D$.  We proceed to show that $T$ preserves the reality of zeros.  To this end, consider $(x^2-1)f(x)$, where $f(x)$ has only real zeros.  We know that differentiation preserves reality of zeros.  Therefore,
\begin{eqnarray*}
D[(x^2-1)f(x)]&=&2x f(x) +(x^2-1)f'(x)\\
&=&[(x^2-1)D+2x]f(x)
\end{eqnarray*}
has only real zeros.  It follows that the operator $(x^2-1)D+2x$ preserves the reality of zeros. The result follows.
\end{proof}

\begin{lemma}\label{k2k1}
$\seq{k^2+k+1}$ is a Legendre multiplier sequence.
\end{lemma}

\begin{proof}
The differential operator associated to $\seq{k^2+k+1}$ is $T=(x^2-1)D^2+2xD+1$.  We show that $T$ preserves the reality of zeros.  By Theorem \ref{3exw} it suffices to show that $T[e^{xw}]$ has no zeros in the region $\Omega := \{(x,w)$ $ \vline $ Im $(x) >0$, Im $(w)>0\}$.  We evaluate
\begin{eqnarray*}
T[e^{xw}]&=&[(x^2-1)D^2+2xD+1][e^{xw}]\\
&=&e^{xw}((xw)^2-w^2+2xw+1).
\end{eqnarray*}
Since $e^{xw}$ is nowhere zero, the zeros of $T[e^{xw}]$ are those of $(xw)^2-w^2+2xw+1$.  Solving
\begin{eqnarray*}
(xw)^2-w^2+2xw+1=0
\end{eqnarray*}
for $w$ we obtain
\begin{eqnarray*}
w_{1,2}=\frac{1}{x \pm 1}= \frac{\bar{x} \pm 1}{| x \pm 1|^2}.
\end{eqnarray*}
It follows that if Im $(x) >0$, then  Im $(w_{1,2}) < 0$.  Therefore $T[e^{xw}]$ has no zeros in the region $\Omega$, and we conclude that $T$ preserves the reality of zeros.  Thus $\seq{k^2+k+1}$ is a Legendre multiplier sequence.
\end{proof}

 %%%%%%%%%%%%Begin new proof %%%%%%%%%%%%%
 \begin{proposition}\label{01}
 If $0<\beta<1$, then the operator $T=\beta+2xD+(x^2-1)D^2$ preserves the reality of zeros. 
 \end{proposition}
 
 \begin{proof} 
 According to Theorem \ref{3exw}, the operator $T$ preserves the reality of zeros as long as the polynomial
 \[
 f(z,w)=(z^2-1)w^2-2zw+\beta=w^2z^2-2wz+\beta-w^2
 \]
 does not vanish whenever Im$(w)>0$ and Im$(z)>0$. Solving $f(z,w)=0$ for $z$ we obtain
  \footnote{We take $\sqrt{(1-\beta)+w^2}$ to be the complex number with the imaginary part of the same sign as that of $w$}
 \[
 z_{1,2}=\frac{2w \pm \sqrt{4w^2-4(w^2)(\beta-w^2)}}{2w^2}=\frac{1\pm\sqrt{(1-\beta)+w^2}}{w}.
 \]
 Suppose first that $w=ki$, where $k>0$. In this case $(1-\beta)+w^2$ is a real number. If this number is positive, then $1\pm\sqrt{(1-\beta)+w^2}$ is real, and hence $z_{1,2}=-i k_{1,2}$. If $(1-\beta)+w^2 <0$, then $\ds{z_{1,2}=-i k\pm\frac{\tilde{k}}{k}}$ for some $\tilde{k} \in \R$, and hence Im$(z_{1,2})<0$.  \\
 We break the rest of the proof into two cases.
 \begin{itemize}
\item[\underline{Case 1:}] {\bf $0<$Arg$(w)<\pi/2$}. Since $0<\beta<1$, we have
 \begin{equation}\label{part1}
\frac{\pi}{2}> \textrm{Arg}(w)>\textrm{Arg}\left(\sqrt{(1-\beta)+w^2}\right)>\textrm{Arg}\left(\sqrt{w^2+1}\right)>0,
 \end{equation}
 and similarly, 
 \begin{equation}\label{part2}
 -\pi <\textrm{Arg}\left(-\sqrt{w^2+1}\right) < \textrm{Arg}\left(-\sqrt{(1-\beta)+w^2}\right)<\textrm{Arg}(-w)<-\frac{\pi}{2}.
 \end{equation}
Equation (\ref{part1}) immediately implies that 
\[
(\dag) \quad -\frac{\pi}{2} < \textrm{Arg}\left(\frac{1+\sqrt{(1-\beta)+w^2}}{w} \right)<0.
\] 
Using equation (\ref{part2}) we deduce that  
 \[
 \textrm{Arg}\left(1-\sqrt{w^2+1}\right) < \textrm{Arg}\left(1-\sqrt{(1-\beta)+w^2}\right)<0, 
 \]
 and consequently,
 \[
 \textrm{Arg}\left(\frac{1-\sqrt{w^2+1}}{w}\right) < \textrm{Arg}\left(\frac{1-\sqrt{(1-\beta)+w^2}}{w}\right)<0. 
 \]
 The identity $\ds{ \frac{1-\sqrt{1+w^2}}{w}=-\frac{w}{1+\sqrt{w^2+1}}}$ together with the fact that \\Arg$(w)>$Arg$(1+\sqrt{1+w^2})$ implies that 
 \[
 (\ddag) \quad -\pi<\textrm{Arg}\left(\frac{1-\sqrt{(1-\beta)+w^2}}{w}\right)<0.
  \]
  Equations $(\dag)$ and $(\ddag)$ together establish that Im$(z_{1,2})<0$.

  \item[\underline{Case 2:}] {\bf $\pi/2<$Arg$(w)<\pi$}.  Since $0<\beta<1$, we have
 \begin{equation}\label{part3}
\pi> \textrm{Arg}\left(\sqrt{w^2+1}\right) > \textrm{Arg}\left(\sqrt{(1-\beta)+w^2}\right) > \textrm{Arg}(w)>\frac{\pi}{2},
 \end{equation}
 and similarly, 
 \begin{equation}\label{part4}
 -\frac{\pi}{2} <\textrm{Arg}(-w)< \textrm{Arg}\left(-\sqrt{(1-\beta)+w^2}\right)<\textrm{Arg}\left(-\sqrt{w^2+1}\right) <0.
 \end{equation}
 Equation (\ref{part3}) implies that 
 \[
  \textrm{Arg}\left(\frac{1+\sqrt{w^2+1}}{w}\right) > \textrm{Arg}\left(\frac{1+\sqrt{(1-\beta)+w^2}}{w}\right)> \textrm{Arg}\left(\frac{1+w}{w}\right). 
\]
We see that 
\[
\textrm{Arg}\left(\frac{1+w}{w}\right)>\textrm{Arg}\left(\frac{1}{w}\right)> -\pi.
\]
The identity $\ds{ \frac{1+\sqrt{1+w^2}}{w}=-\frac{w}{1-\sqrt{w^2+1}}}$ together with the fact that \\Arg$(-w)<$Arg$(1-\sqrt{1+w^2})$ implies that 
 \[
 (*) \quad -\pi<\textrm{Arg}\left(\frac{1+\sqrt{(1-\beta)+w^2}}{w}\right)<0.
  \]
  Using equation (\ref{part4}) we deduce that  
 \[
 \textrm{Arg}\left(1-w\right) < \textrm{Arg}\left(1-\sqrt{(1-\beta)+w^2}\right)<0, 
 \]
 and consequently,
 \[
 \textrm{Arg}\left(\frac{1-w}{w}\right) < \textrm{Arg}\left(\frac{1-\sqrt{(1-\beta)+w^2}}{w}\right)<0. 
 \]
Note that
 \[
 -\pi= \textrm{Arg}\left(\frac{-w}{w}\right)<\textrm{Arg}\left(\frac{1-w}{w}\right),
 \]
 which implies that
 \[
 (**) \quad -\pi<\textrm{Arg}\left(\frac{1-\sqrt{(1-\beta)+w^2}}{w}\right)<0.
 \]
  Equations $(*)$ and $(**)$ together establish that Im$(z_{1,2})<0$.

  \end{itemize}
 \end{proof}
 %%%%%%%%%%%%%%%%%% End New Proof%%%%%%%%%%%%%%%

\noindent We have thus proved the following.

\begin{proposition}\label{a=1}
If $\alpha =1$ and $\beta \in [0,1]$, then $\seq{k^2+\alpha k+\beta}$ is a Legendre multiplier sequence.
\end{proposition}

Propositions \ref{a1} and \ref{a=1} together characterize all quadratic Legendre multiplier sequences.

\begin{theorem}\label{quadkill}
$\seq{k^2+\alpha k+\beta}$ is a Legendre multiplier sequence if and only if $\alpha = 1$ and $\beta \in [0,1]$.
\end{theorem}

%%%%%%%%%%GEOMETRIC
\subsection{Geometric Sequences}

We now look at the geometric sequences $\{r^k\}_{k=0}^{\infty}$, where $r \in \R \setminus \set{0}$.  It is known that sequences of this form are classical multiplier sequences, but are Hermite multiplier sequences if and only if $|r| \ge 1$.  It has also been shown by Forg\'{a}cs and Piotrowski \cite{fp} that the only such Laguerre multiplier sequence is the constant sequence $\{1\}_{k=0}^{\infty}$.

\begin{theorem} \label{geom}
If $r \neq 0$, then $\{r^k\}_{k=0}^{\infty}$ is a Legendre multiplier sequence if and only if $|r|=1$.
\end{theorem}

\begin{proof} Assume first that $|r|=1$. If $r=1$, then the sequence $\seq{r^k}$ is trivial. Suppose now that  $\ds{p(x)=\sum_{k=0}^{n} a_i \Le_i(x)}$ has only real zeros. Then 
\[
p(-x)=\sum_{k=0}^{n} a_i \Le_i(-x)=\sum_{k=0}^{n} a_i (-1)^i \Le_i(x)
\]
also has only real zeros, which implies that $\seq{-1^k}$ is a Legendre multiplier sequence.\\
For the converse assume that $|r| \neq 1$. If $0<|r|<1$, the sequence $\seq{r^k}$ is not a Hermite multiplier sequence, and hence by the results in Section \ref{prop} it cannot be a Legendre multiplier sequence. \\
If $|r| > 1$, we apply the sequence $\{r^k\}_{k=0}^{\infty}$ to the polynomial $p(x)=(x+r)^4$. We thus obtain the polynomial
\[
\widetilde{p}(x)=\frac{16}{5}+\frac{32}{5}rx+\frac{16}{7}r^2(-1+3x^2)+\frac{4}{5}r^3(-3x+5x^3)+\frac{1}{35}r^4(3-30x^2+35x^4).
\]
The discriminant of $\widetilde{p}(x)$ with respect to $x$ is given by
\begin{eqnarray*}
\Delta=\frac{16384}{10504375}(44044r^{12}-147576r^{14}+180624r^{16}-96991r^{18}+22329r^{20}-2565r^{22}+135r^{24}).
\end{eqnarray*}
By Theorem \ref{qd} part (1)(b) we conclude that $\widetilde{p}(x)$ has complex roots. Hence $\seq{r^k}$ is not a Legendre multiplier sequence when $|r|>1$. 

\end{proof}

%%%%%%%%%%%%%%%%%%%CONCLUSION%%%%%%%%%%%%%%

\section{Open questions} \label{concl}
In this paper we partially classified Legendre multiplier sequences.  We exhibited general properties of Legendre multiplier sequences and showed that the Legendre multiplier sequences are a strict subset of the Hermite multiplier sequences. We then gave a complete characterization of linear, quadratic, and geometric Legendre multiplier sequences. We conclude the paper with a short list of open questions:
\begin{enumerate}
\item We know that a large class of cubic sequences fail to be Legendre multiplier sequences. There is overwhelming numerical evidence to suggest that there are {\it no} cubic Legendre sequences, a conjecture based on the complete lack of linear Legendre multiplier sequences.
\item As a generalization of the first problem, we propose that there are no Legendre multiplier sequences of {\it any odd degree}. This claim is supported by the differential equation
\[
(\star) \quad (x^2+1) \Le_n''(x)-2x \Le_n'(x)=n(n+1) \Le_n(x),
\]
where the coefficient of the non-differentiated term is of even degree in $n$, although we have no numerical evidence beyond the cubic sequences.
\item Given a sequence of real numbers $\seq{\gamma_k}$, consider the operator $\Gamma$ defined by $\Gamma[\Le_n(x)]=\gamma_n \Le_n(x)$. Both $\Gamma$ and the operator $T=(x^2-1)D^2-2x D$ (as in the left hand side of equation ($\star$)) are diagonal with respect to the Legendre basis and hence $\Gamma T=T \Gamma$. Can one characterize all differential operators which commute with $T$, and give sufficient and/or necessary conditions for such operators to correspond to Legendre multiplier sequences?  (This approach was suggested by David Cardon at Brigham Young University)
\item It is known that the falling factorial sequence 
\[
\left\{\frac{\Gamma(k+1)}{\Gamma(k-n+1)}  \right\}_{k=0}^{\infty} = \{k(k-1)\cdots(k-n+1)\}_{k=0}^{\infty}
\]
is a classical, Hermite and generalized Laguerre multiplier sequence for every $n \in \N$ but it is not a Legendre multiplier sequence. We suspect however that the sequence
\[
\left\{\frac{\Gamma(k(k+1)+1)}{\Gamma(k(k+1)-n+1)}  \right\}_{k=0}^{\infty}
\]
is a Legendre multiplier sequence, although none of the methods we are familiar with yield a proof of this fact.
\end{enumerate}

%
%

%\end{doublespace}
\end{document}